\documentclass[11pt]{article}
\usepackage{amsmath, amssymb, amsfonts, amsthm, graphicx, mathrsfs, relsize, euscript, stmaryrd, float}
\usepackage[labelsep=period, font=footnotesize, labelfont=bf]{caption}
\newtheorem{theorem}{Theorem}[section]
\newtheorem{lemma}[theorem]{Lemma}

\newtheorem{thm}[theorem]{Theorem}
\newtheorem{predefinition}[theorem]{{\bf Definition}}
\newenvironment{definition}{\begin{predefinition}\rm{\hspace{-0.5 em}{\bf}}}{\end{predefinition}}
\newtheorem{prequestion}[theorem]{{\bf Question}}
\newenvironment{question}{\begin{prequestion}\rm{\hspace{-0.5 em}{\bf}}}{\end{prequestion}}
\newtheorem{preremark}[theorem]{{\bf Remark}}
\newenvironment{remark}{\begin{preremark}\rm{\hspace{-0.5 em}{\bf}}}{\end{preremark}}

\DeclareMathAlphabet{\mathbbmsl}{U}{bbm}{m}{sl}

\DeclareMathAlphabet{\mathpzc}{OT1}{pzc}{m}{it}

\makeatletter
\DeclareRobustCommand\widecheck[1]{{\mathpalette\@widecheck{#1}}}
\def\@widecheck#1#2{%
    \setbox\z@\hbox{\m@th$#1#2$}%
    \setbox\tw@\hbox{\m@th$#1%
       \widehat{%
          \vrule\@width\z@\@height\ht\z@
          \vrule\@height\z@\@width\wd\z@}$}%
    \dp\tw@-\ht\z@
    \@tempdima\ht\z@ \advance\@tempdima2\ht\tw@ \divide\@tempdima\thr@@
    \setbox\tw@\hbox{%
       \raise\@tempdima\hbox{\scalebox{1}[-1]{\lower\@tempdima\box
\tw@}}}%
    {\ooalign{\box\tw@ \cr \box\z@}}}
\makeatother

\title{On  $K_{2,t}$-bootstrap percolation\\[4mm]}

\author{
{M.R.  Bidgoli} \quad  {A. Mohammadian}   \quad  {B. Tayfeh-Rezaie}\\[4mm]
School of Mathematics,\\
Institute for Research in Fundamental Sciences\,(IPM),\\
P.O. Box 19395-5746, Tehran, Iran\\[3mm]
\textsf{\{bd, ali\_m, tayfeh-r\}}@\textsf{ipm.ir}\\[8mm]}

\date{}

\begin{document}

\maketitle

\begin{abstract}
Given   two  graphs $G$ and $H$, it is  said  that  $G$  percolates in  $H$-bootstrap process if one could join all the nonadjacent pairs of vertices of $G$ in some order such that a  new copy of $H$ is created at each step. Balogh, Bollob\'as and Morris in 2012  investigated   the   threshold of $H$-bootstrap percolation  in the Erd\H{o}s$\text{\bf--}$R\'enyi model for the complete graph  $H$  and proposed the similar problem for  $H=K_{s,t}$,    the   complete bipartite graph. In this paper, we provide  lower and upper bounds on  the  threshold of $K_{2, t}$-bootstrap percolation. In addition,      a    threshold function is derived for     $K_{2, 4}$-bootstrap percolation. \\[3mm]
\noindent {\bf AMS Mathematics Subject Classification\,(2010):}    05C80, 60K35. \\[6mm]
\end{abstract}

\section{Introduction}

Bootstrap percolation  on graphs has been extensively investigated    in several diverse fields such as  combinatorics,  probability theory, statistical physics and social sciences.  Many different  models of   bootstrap percolation  have    been  defined and  studied in the literature    including  the    $r$-neighbor bootstrap percolation and the  majority bootstrap percolation.  In this paper, we deal  with the $H$-bootstrap percolation  whose study   was initiated  in 2012  by Balogh, Bollob\'as and Morris    \cite{bal}.  Roughly speaking, for  two given  graphs $G$ and $H$, we say  that  $G$  percolates in the  $H$-bootstrap process   if it is possible to join   all the nonadjacent pairs of vertices of $G$ in some order such that a  new copy of $H$ is created at each step.  The concept is closely  related  to the notion of  weak saturation  that was    introduced in 1968  by   Bollob\'as   \cite{bol}.  The $H$-bootstrap percolation has been  also  studied  by other researches, see  \cite{ang, bolp, gun, kol}.

Throughout   this  paper, all graphs    are assumed to be finite,   undirected,  and without loops or multiple edges.
For   a graph $G$, we denote  the vertex set and the edge set of   $G$  by $V(G)$ and $E(G)$, respectively. For    given    graphs $G$ and $H$, we  associate  the graph $\widehat{G}_H$ obtained from  the following process:   Let $G_0=G$  and for $i=1, 2, \ldots$  define $G_i$ as the graph with   vertex set $V(G)$ and edge set  $E(G_{i-1})\cup E_i$, where $E_i$ is the set of all edges in  the complement of  $G_{i-1}$ such that adding each of them to $G_{i-1}$ creates a new copy of $H$. Define $\widehat{G}_H$ as the graph with   vertex set $V(G)$ and edge set  $\mathsmaller{\bigcup}_{i\geqslant0}E(G_i)$.   We say that $G$ {\sl percolates in the $H$-bootstrap process} if $\widehat{G}_H$  is a complete graph.

For two positive   real valued functions $f$ and $g$  defined on positive integers, we write $f=\mathrm{O}(g)$ (respectively, $f=\Omega(g)$) if there exists    a  positive constant   $c$ such that $f(n)\leqslant cg(n)$ (respectively, $f(n)\geqslant cg(n)$) for any $n$ large enough. Further, we write  $f=\Theta(g)$ if $f=\mathrm{O}(g)$ and  $f=\Omega(g)$. Finally, we write $f\ll g$   (respectively,  $f\gg g$) if $\lim_{n\rightarrow\infty}\tfrac{f(n)}{g(n)}$ equals $0$   (respectively, $\infty$).

For a positive  integer $n$ and a function $p$   defined on positive integers with values  in   $[0, 1]$, we denote by $\mathbbmsl{G}(n, p)$  the
probability space of all graphs on a fixed vertex set of size $n$ where every two  distinct    vertices  are adjacent   independently with probability $p(n)$.
In the literature,  $\mathbbmsl{G}(n, p)$ is known as   the Erd\H{o}s$\text{\bf--}$R\'enyi model for random graphs.
A function $\widecheck{p}$ is a {\sl threshold} for   a sequence   $\mathscr{E}_n$    of    events   in $\mathbbmsl{G}(n, p)$    if
$$\mathlarger{\lim}_{n\rightarrow\infty}\mathbbmsl{P}(\mathscr{E}_n)=
\left\{\begin{array}{ll}
0,  &   \mbox{ if } p\ll \widecheck{p}\mbox{;}\\
\vspace{-1mm}\\
1, &  \mbox{ if } p\gg \widecheck{p}\mbox{.}
\end{array}\right.$$
We say  that  $\mathscr{E}_n$ holds {\sl with high probability}  if $\lim_{n\rightarrow\infty}\mathbbmsl{P}(\mathscr{E}_n)=1$.
By    a result of    Bollob\'as  and     Thomason \cite{bolt}, for any  graph $H$,   $$p_c(n; H)=\inf\Big\{p\in[0, 1] \, \Big| \, \mathbbmsl{P}\big(\text{$\mathbbmsl{G}(n, p)$  percolates in $H$-bootstrap process}\big)\geqslant\tfrac{1}{2}\Big\}$$ is  a threshold function for  $H$-bootstrap percolation.

Denote the  complete graph on $r$ vertices and  the complete bipartite graph with   part  sizes $s$ and $t$ by $K_r$ and $K_{s, t}$, respectively. Balogh, Bollob\'as and Morris in  \cite{bal}  studied     $H$-bootstrap percolation on $\mathbbmsl{G}(n, p)$. They  proved
for any   fixed integer  $r\geqslant4$ and  any   sufficiently large $n$ that
$$\frac{n^{-\lambda}}{2\text{\sl e}\log n}\leqslant p_c(n; K_r)\leqslant n^{-\lambda}\log n,$$ where $\lambda=\tfrac{2r-4}{r^2-r-4}$.
One of the   open   problems   posed in that  paper is   the  determination of   $p_c(n; K_{s,t})$. We know that
$$p_c(n; K_{1,t})=\Theta\hspace{-1mm}\left(n^{-\tfrac{t}{t-1}}\right) \mbox{  for any }  t\geqslant2$$ and  $$p_c(n; K_{2,2})=p_c(n; K_{2,3})=\frac{\log n}{n}+\Theta\hspace{-1mm}\left (\frac{1}{n}\right),$$   according to some  results in    \cite{bal}.  In this   paper, we examine  $p_c(n; K_{2,t})$ for    $t\geqslant4$. We present lower and upper bounds on $p_c(n; K_{2,t})$ and moreover, we prove that $p_c(n; K_{2,4})=\Theta(n^{-10/13})$.

Let us fix some notation and terminology.   For a  graph $G$ and a   subset   $S$ of $V(G)$, we denote the induced subgraph of $G$ on $S$  by $G[S]$. For a vertex     $v\in V(G)$, we set   $N_G(v)=\{x\in V(G) \, | \, v \text{   is adjacent to } x\}$  and    $N_G[v]=N_G(v)\cup\{v\}$. The {\sl degree} of  a vertex  $v\in V(G)$,  denoted by $\deg_G(v)$,  is defined  as $|N_G(v)|$.  A graph  $G$   is a {\sl complete split graph} if one  can partition $V(G)$   into an independent   set $I$ and a clique $C$ such that each  vertex in $I$  is adjacent to each  vertex in $C$.

\section{The  upper bound}

In this  section, we assume that  $t$ is an integer at least $4$ and  we reserve    $\widehat{G}$ for  the graph obtained from a graph $G$ in     $K_{2,t}$-bootstrap process. We will  obtain  an upper bound on  $p_c(n; K_{2,t})$. More precisely, we will   establish   that
$$p_c(n; K_{2,t})=\mathrm{O}\hspace{-1mm}\left(n^{-\tfrac{1}{\eta(t)}}\right),$$ where
$$\eta(t)=\left\{\begin{array}{ll}
\mathlarger{\frac{6t^2-14t+12}{3t^2-4t+8}},  &   \mbox{ if $t$ is even;}\\
\vspace{-1mm}\\
\mathlarger{\frac{2t^2-4t+2}{t^2-t+2}}, &  \mbox{ if $t$ is odd.}
\end{array}\right.$$
Recall that the {\sl density} of a graph  $G$ is defined as $$d(G)=\frac{|E(G)|}{|V(G)|},$$ and the    {\sl maximum subgraph density} of $G$ as
$$m(G)=\max\Big\{d(H) \, \Big| \,  H  \mbox{ is a subgraph of }  G\Big\}.$$
In our proofs, we frequently use  the  following theorem.

\begin{thm}\label{bolo}
{\rm (Bollob\'as \cite{bo})}
Let $H$ be a fixed graph with at least one edge. Then $n^{-1/m(H)}$ is a threshold for the property that $\mathbbmsl{G}(n, p)$
contains a copy of $H$ as a subgraph.
\end{thm}

The following lemma is easily   obtained from  the  definition of  $K_{2,t}$-bootstrap process.

\begin{lemma}\label{nei}
Let $G$ be a graph and let $x,y \in V(G)$ with $|N_G(x) \cap N_G(y)| \geqslant t-1$. Then $N_{\widehat{G}}(x)\setminus\{y\}=N_{\widehat{G}}(y)\setminus\{x\}$.
\end{lemma}

\begin{lemma}\label{t-1}
Let   $G$  be  a connected graph containing    a copy of $K_{t-1,t-1}$ as a subgraph. Then $\widehat{G}$  is either a complete graph, a complete bipartite graph   or a complete split graph with   the   clique part of size   $t-1$.
\end{lemma}

\begin{proof}
We consider the relation $\thickapprox$ on $V(\widehat{G})$  as follows:
$$x\thickapprox y \mbox{\quad if \quad} N_{\widehat{G}}(x)\setminus\{y\}=N_{\widehat{G}}(y)\setminus\{x\}.$$
It is straightforward  to check that $\thickapprox$ is an equivalence relation. Furthermore, each equivalence class is either an independent set or a  clique  and between any  two classes either there is no edge or all  possible  edges are present.

Let $H$ be a copy of $K_{t-1,t-1}$ in $G$ with bipartition $V(H)=A\cup B$.
It follows from  Lemma \ref{nei} that   $A$, and similarly $B$,  is contained in some  equivalence class.  Let $\text{\bf[}A\text{\bf]}$ and $\text{\bf[}B\text{\bf]}$ be the equivalence classes containing $A$ and $B$, respectively. Note  that  $\text{\bf[}A\text{\bf]}$ and $\text{\bf[}B\text{\bf]}$ are not necessary distinct.   We show that $V(G)=\text{\bf[}A\text{\bf]}\cup\text{\bf[}B\text{\bf]}$ which implies the assertion of the lemma.  By contradiction, suppose that $V(G)\neq\text{\bf[}A\text{\bf]}\cup\text{\bf[}B\text{\bf]}$. As $G$ is connected, there is  a vertex $v\notin\text{\bf[}A\text{\bf]}\cup\text{\bf[}B\text{\bf]}$ with a  neighbor  in  $\text{\bf[}A\text{\bf]}$ or $\text{\bf[}B\text{\bf]}$, say $\text{\bf[}A\text{\bf]}$. Note  that $v$ is adjacent to the whole $\text{\bf[}A\text{\bf]}$. Therefore,    $|N_G(v)\cap N_G(w)|\geqslant t-1$ for   arbitrarily chosen   vertex  $w\in B$. Using   Lemma \ref{nei}, $v\thickapprox w$ and hence  $v\in\text{\bf[}B\text{\bf]}$, a contradiction.

Now, assume that $\widehat{G}$ is a complete split graph with the   independent part $I$ and the   clique part $C$. Note that $I$ and  $C$ are   the equivalence classes of  $\thickapprox$.    If $|C|\geqslant t$, then every two  vertices  $x\in I$ and $y\in C$ have at least $t-1$ common neighbors in $C$. Hence,   Lemma \ref{nei} yields that  $x\thickapprox y$,  a contradiction.
\end{proof}

\begin{definition}\label{ht}
For two  positive integers $r$ and  $s$, consider $s$ copies of $K_{2,r}$ and let    $\{u_i, u'_i\}$ be  a   part  of size $2$    in   the   $i$th copy. We denote  by  $\mathpzc{G}_r(u; u_1, \ldots, u_s)$  the graph obtained  by     identifying  all $u'_1, \ldots, u'_s$ to a single vertex $u$. For instance,  the graph  $\mathpzc{G}_4(u; u_1, u_2, u_3)$ is depicted in  Figure \ref{f1f}. For an integer   $t\geqslant4$,   let $r=\lfloor (t-1)/2\rfloor$ and $s=t-1-r$. We define   $\mathpzc{H}_t$   as  the graph made   of  the vertex disjoint graphs  $\mathpzc{G}_{t-1}(u; u_1, \ldots, u_r)$,  $\mathpzc{G}_{s-1}(v; v_1, \ldots, v_s)$ and  $\mathpzc{G}_{r-1}(w; w_1, \ldots, w_{t-2})$ by joining $u$ to   $v, v_1, \ldots, v_s$ and  $v$ to  $w, w_1, \ldots, w_{t-2}$. For example, the graph    $\mathpzc{H}_8$ is   shown  in Figure \ref{f2f}.
\begin{figure}[H]
\centering
\includegraphics[width=0.4\textwidth]{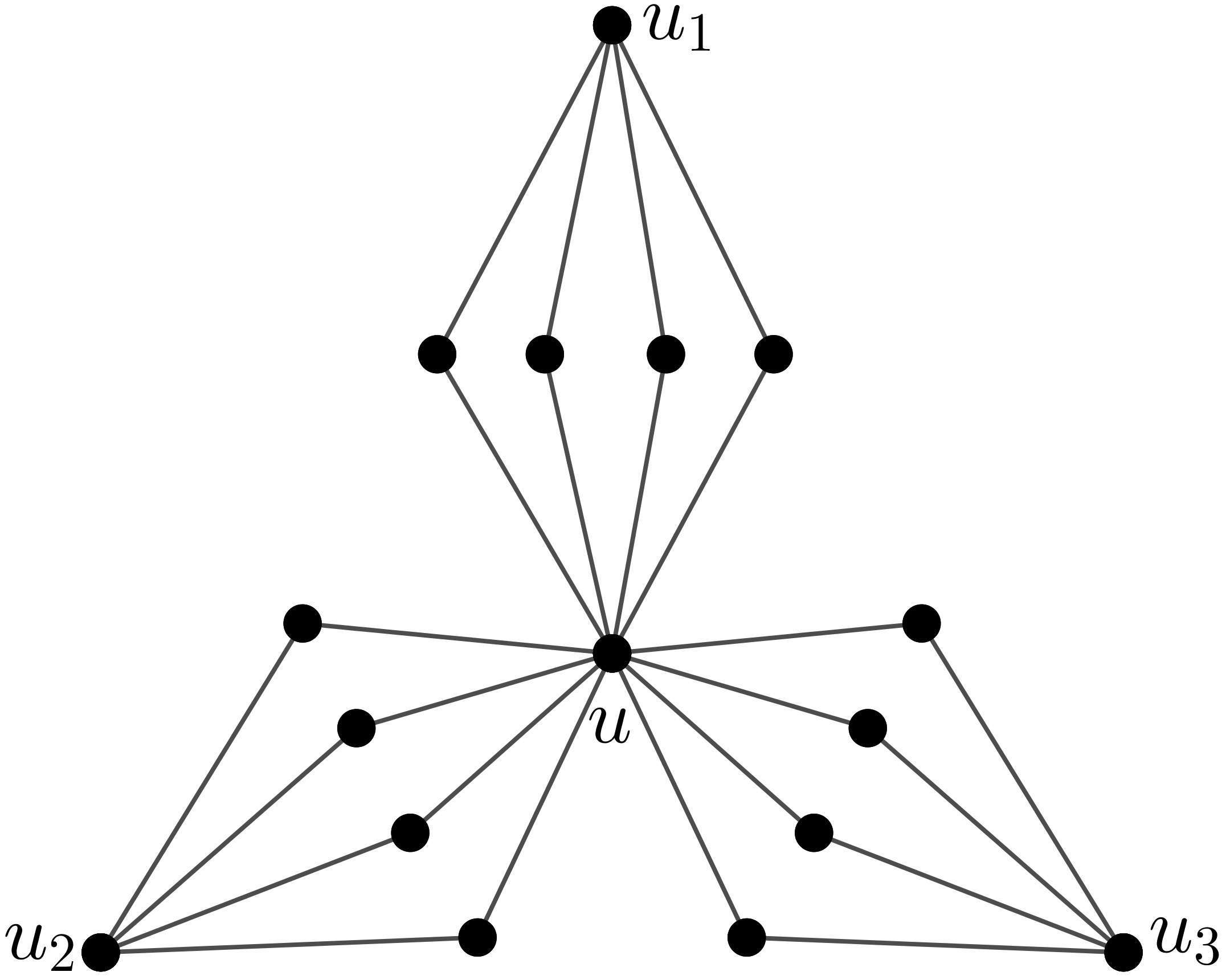}
\caption{The graph $\mathpzc{G}_4(u; u_1, u_2, u_3)$.}\label{f1f}
\end{figure}
\begin{figure}[H]
\centering
\includegraphics[width=1\textwidth]{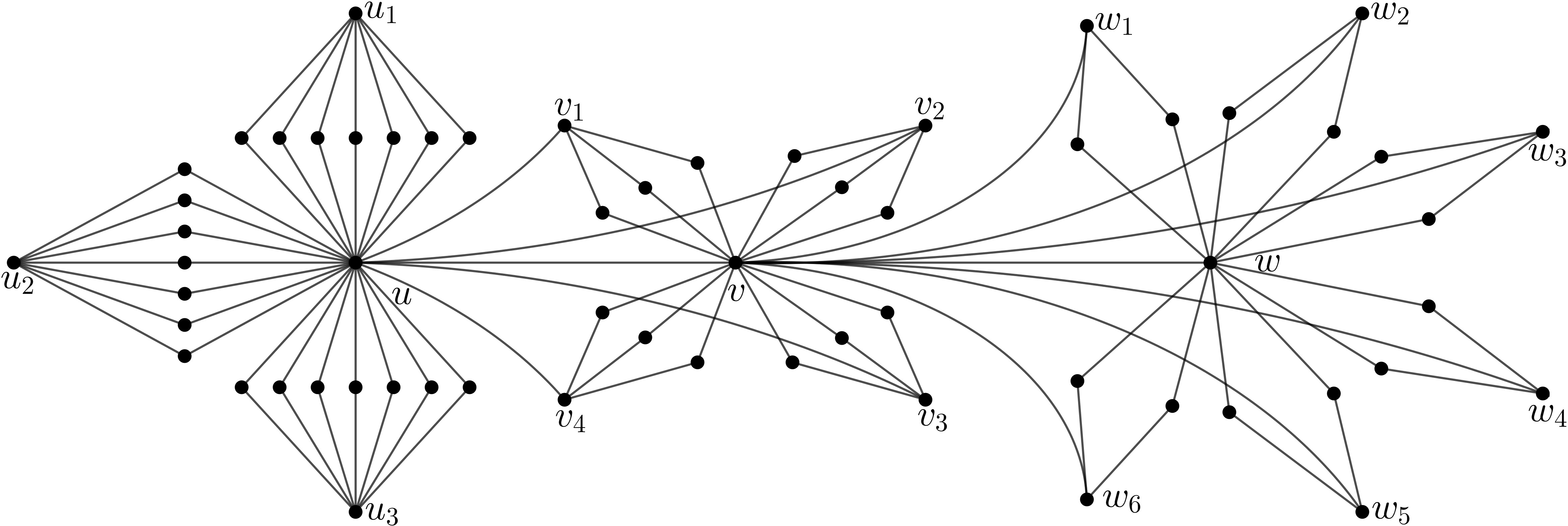}
\caption{The graph $\mathpzc{H}_8$.}\label{f2f}
\end{figure}
\end{definition}

\begin{thm}\label{density}
For any $t\geqslant4$,  $m(\mathpzc{H}_t)=\eta(t)$.
\end{thm}

\begin{proof}
For convenience,  let  $G=\mathpzc{H}_t$ and  $m=m(G)$. Assume that  $H$ is  a subgraph of $G$  with minimum possible number of vertices satisfying   $d(H)=m$.  We need to prove the following facts about  $H$.

\noindent{\bf\textsf{Fact 1.}} The minimum degree of  $H$ is    $2$.

Since  $t\geqslant 4$ and $G$ contains a copy of $K_{2,t-1}$, we find  that  $m>1$. For each vertex  $v\in V(H)$, it follows from    $d(H-v)\leqslant d(H)$ that $\deg_H(v)\geqslant m$. Therefore, the minimum degree of  $H$ is  at least    $2$.   On the other hand, it is easily  seen that   $G$ has no subgraph with the minimum degree more than $2$, implying the desired property.

\noindent{\bf\textsf{Fact 2.}}  For every  two  distinct vertices  $x,y \in V(H)$,  $N_G(x)\cap N_G(y)\subseteq V(H)$.

For  a vertex $v\in V(H)$ with $\deg_H(v)=2$, it follows from the minimality of $|V(H)|$ that     $d(H-v)<d(H)$ which in turn implies  that $m<2$. Now,
if a vertex $x\in V(G)\setminus V(H)$ is adjacent to at least two vertices  in $V(H)$, then it follows from  $m<2$   that $d(G[\{x\}\cup V(H)])>d(H)$, a contradiction.   This shows the correctness of Fact 2.

\noindent{\bf\textsf{Fact 3.}}  If $u_{i_0}\in V(H)$ for some $i_0$, then $u$ and all $u_i$ are contained in $V(H)$.  Similar statements hold for $v_i$ and $w_i$.

By contradiction, without   loss of generality, assume that  $u_1\in V(H)$ and  $u_2\notin V(H)$.  Facts 1 and 2   imply that  $\{u\}\cup N_G(u_1)\subseteq V(H)$ and $N_G(u_2)\cap V(H)=\varnothing$.  The   minimality of $|V(H)|$ forces   that     $d(H-N_G[u_1])<d(H)$ which in turn   yields  that   $m<2(t-1)/t$. This shows    that  $d(G[N_G[u_2]\cup V(H)])>d(H)$, a contradiction.       The proofs  for   $v_i$ and $w_i$ are similar.

Applying   Facts 1{\bf--}3 and noting that  $H$ is an induced subgraph of $G$, we are left with only seven candidates  for $V(H)$ as described below. Letting
$$A=\bigcup_{i=1}^rN_G[u_i], \, \,   B=\bigcup_{i=1}^sN_G[v_i] \, \,   \text{ and } \, \,   C=\bigcup_{i=1}^{t-2}N_G[w_i],$$
where $r, s$ are as defined  in Definition \ref{ht},   $V(H)$ is equal to one of the subsets
$$\{u\}\cup A, \{v\}\cup B, \{w\}\cup C, \{v\}\cup A \cup B, \{w\}\cup B\cup C,  \{u, w\}\cup A\cup C, \{w\}\cup A\cup B\cup C.$$
It is a  matter of straightforward  calculation to show  that,   among the      subgraphs of $G$ induced on these seven subsets, the maximum density occurs in $G[\{u\}\cup A]$ if $t$ is odd and in $G[\{v\}\cup A\cup B]$, otherwise. Since $$d\big(G[\{u\}\cup A]\big)=\frac{2t^2-4t+2}{t^2-t+2} \, \,   \mbox{ and } \, \, d\big(G[\{v\}\cup A\cup B]\big)=\frac{6t^2-14t+12}{3t^2-4t+8},$$   the proof is complete.
\end{proof}

Now we are ready to prove our upper  bound on $p_c(n; K_{2,t})$.

\begin{thm}\label{upper}
For any fixed integer  $t\geqslant4$, $$p_c(n; K_{2,t})=\mathrm{O}\hspace{-1mm}\left(n^{-\tfrac{1}{\eta(t)}}\right).$$
\end{thm}

\begin{proof}
Let  $G\thicksim\mathbbmsl{G}(n, p)$ and  $p\gg n^{-1/\eta(t)}$. Using    Theorems \ref{bolo} and  \ref{density}, $G$  with high probability contains a copy  of
$\mathpzc{H}_t$, say $H$. Applying   Lemma \ref{nei}, $N_{\widehat{H}}(u)\setminus\{u_i\}=N_{\widehat{H}}(u_i)\setminus\{u\}$ for $i=1, \ldots, r$,  where $r$ is as  defined  in Definition \ref{ht}. This shows that $u_i$ is adjacent to $v, v_1, \ldots, v_s$ for any  $i$. Hence, $|N_{\widehat{H}}(v)\cap N_{\widehat{H}}(v_j)|\geqslant  t-1$  for $j=1, \ldots, s$,  where $s$  is as   defined  in Definition \ref{ht}. Again, it   follows from    Lemma \ref{nei} that $N_{\widehat{H}}(v)\setminus\{v_j\}=N_{\widehat{H}}(v_j)\setminus\{v\}$  for any $j$.  This shows that $v_j$ is adjacent to $w, w_1, \ldots, w_{t-2}$ for any  $j$. Therefore, for any $k$,  $|N_{\widehat{H}}(w) \cap N_{\widehat{H}}(w_k)|\geqslant t-1$ which implies that   $N_{\widehat{H}}(w)\setminus\{w_k\}=N_{\widehat{H}}(w_k)\setminus\{w\}$  by    Lemma \ref{nei}. This shows that $\widehat{H}$,   and in turn  $\widehat{G}$,   contains   a copy of $K_{t-1, t-1}$. Since  $p\gg  \log n/n$, $G$ with high probability  is   connected and nonbipartite by  \cite[Theorem\,4.1]{fri} and Theorem \ref{bolo}. So,     Lemma \ref{t-1} yields  that $\widehat{G}$ is either  a  complete split graph  or a complete graph. If  $\widehat{G}$ is a complete split graph with the  independent part $I$ and the  clique part $C$, then each vertex in $I$  has at least $np/2$ neighbors in $C$ with high probability \cite[Theorem\,3.4]{fri}. Thus,    $|C|\geqslant t$ which contradicts Lemma \ref{t-1}. Consequently,  $\widehat{G}$ is complete and the result follows.
\end{proof}

It is natural to ask whether the upper  bound given in Theorem \ref{upper} is   in fact a threshold. So,  we pose the following question.

\begin{question}\label{que}
Is it true that $p_c(n; K_{2,t})=\Theta\hspace{-0.7mm}\left(n^{-1/\eta(t)}\right)$ for any $t\geqslant4$?
\end{question}

For $t=4$,  an   affirmative answer to  Question \ref{que} is given by the following theorem.

\begin{thm}\label{1310}
$p_c(n; K_{2,4})=\Theta\hspace{-0.7mm}\left(n^{-{10}/{13}}\right)$.
\end{thm}

\begin{proof}
By Theorem \ref{upper},  it suffices  to prove that  $p_c(n; K_{2,4})=\Omega(n^{-{10}/{13}})$.   If $G\thicksim\mathbbmsl{G}(n, p)$ with  $p\ll n^{-{10}/{13}}$, then  Theorem \ref{bolo}   and   the union bound theorem   imply  that $G$ with high probability contains no  bounded subgraph $H$ with $m(H)\geqslant\tfrac{13}{10}$. So,  in order to prove  $p_c(n; K_{2,4})=\Omega(n^{-{10}/{13}})$, it is enough  to show that any graph with no bounded subgraph $H$ satisfying $m(H)\geqslant\tfrac{13}{10}$  does not percolate  in  $K_{2,4}$-bootstrap process.

Fix a graph   $G$    without any   bounded subgraph $H$ with $m(H)\geqslant\tfrac{13}{10}$.
We define a sequence  $F_1, F_2, \ldots$  of  vertex disjoint subgraphs of $G$  by the following procedure.  At each step  $i$,  we look for  a  copy of $K_{2,3}$ in $H_i=G-\mathsmaller{\bigcup}_{k=1}^{i-1}V(F_k)$. If there is no such a copy, we finish the  procedure.  Otherwise, we choose  a copy $L$ of $K_{2,3}$ in $H_i$  with bipartition $A$ and   $B$, where $|A|=2$. At the beginning of step $i$,  we set $F_i=G[V(L)]$,  $\mathscr{A}_i=\{A\}$, $\mathscr{B}_i=B$, $\ell_i=\ell'_i=0$.

If there exist   two adjacent  vertices $u, v\in V(H_i)\setminus V(F_i)$ such that   $N_G(u)\cap A\neq\varnothing$ and $N_G(v)\cap B\neq\varnothing$, then we do the following: First choose a vertex $w\in N_G(v)\cap B$. Then,   update   $F_i$, $\mathscr{A}_i$,   $\mathscr{B}_i$ to  $G[V(F_i)\cup\{u, v\}]$,  $\mathscr{A}_i\cup\{\{u, w\}\}$,  $(\mathscr{B}_i\cup\{v\})\setminus\{w\}$, respectively,    and    increment $\ell_i$.

Otherwise, perform   the following iterative subprocedure as long as possible:  Find    three  distinct   vertices $u, v, w\in V(H_i) \setminus V(F_i)$ such that  $w\in N_G(u)\cap N_G(v)$ and both    $N_G(u), N_G(v)$ intersect an element  $P\in\mathscr{A}_i$. Add $\{u, v\}$ to $\mathscr{A}_i$ and   $w$ to $\mathscr{B}_i$. In addition, update   $F_i$ to $G[V(F_i)\cup\{u, v, w\}]$ and increment $\ell'_i$.

We now state some properties of $F_i$. According to the procedure,   $|V(F_i)|=2\ell_i+3\ell'_i+5$ and $|E(F_i)|\geqslant3\ell_i+4\ell'_i+6$. If $4\ell_i+\ell'_i\geqslant5$, then   $d(F_i)\geqslant\tfrac{13}{10}$ which contradicts our assumption on $G$. Thus, $|V(F_i)|$ is bounded. The following properties of $F_i$ are also proved using similar arguments.

\noindent{\bf\textsf{Fact 1.}}  $|E(F_i)|=3\ell_i+4\ell'_i+6$.

\noindent{\bf\textsf{Fact 2.}}  There is no edge between   $V(F_i)$ and $V(F_j)$ if    $i\neq j$.

\noindent{\bf\textsf{Fact 3.}} There exists  at most one vertex $x$ such that $N_G(x)$  intersects   both   $V(F_i)$ and $V(F_j)$ whenever $i\neq j$.

We consider  an auxiliary graph $G'$ obtained from $G$ as follows:  For every integer  $i$ and every  element  $\{a,b\}\in\mathscr{A}_i$,   join  $a$ to all vertices in $N_G(b)\setminus N_G(a)$ and     $b$ to all vertices in $N_G(a)\setminus N_G(b)$.  We claim that $\widehat{G}=G'$.   Since any  pair in  $\mathscr{P}=\mathsmaller{\bigcup}_{i\geqslant0}\mathscr{A}_i$ is an  independent set  in  $G'$ by Fact 1,  the claim   concludes that $G$ does not percolate in  $K_{2,4}$-bootstrap process.

In order to prove the claim, it is enough to show  that  there is no pair     $\{x, y\}\notin\mathscr{P}$ with  $|N_{G'}(x)\cap N_{G'}(y)|\geqslant3$. Towards a  contradiction, suppose that there exists such a  pair   $\{x, y\}$. Let $S_1=\{x, y\}$ and  fix a subset  $S_2\subseteq N_{G'}(x)\cap N_{G'}(y)$ such that  $|S_2|\in\{3, 4\}$ and $|P\cap S_2|\in\{0, 2\}$ for each $P\in\mathscr{P}$.  Put  $S=S_1\cup S_2$.  By  Facts 2 and 3,  $V(F_i)\cap S=\varnothing$ for all $i$ except one, say $i_0$. We drop the subscript $i_0$ from  $F_{i_0}, \mathscr{A}_{i_0}, \mathscr{B}_{i_0}, \ell_{i_0}, \ell'_{i_0}$ in what follows.

First  we assume that $S\setminus V(F)\neq\varnothing$. Set $\alpha=|S_1\setminus V(F)|$, $\beta=|S_2 \setminus V(F)|$,   $\gamma=|S_2\cap\mathscr{B_1}|$  and    $\delta=|\{P\in\mathscr{A} \, |  \, \,   |P\cap S_2|=2\}|$. Clearly, $\beta+\gamma+2\delta=|S_2|$. Letting   $Z=G[S\cup V(F)]$, we   have   $|V(Z)|=\alpha+\beta+2\ell+3\ell'+5$ and $|E(Z)|\geqslant\alpha\gamma+\alpha\delta+2\beta+3\ell+4\ell'+6$. It follows from $d(Z)<\tfrac{13}{10}$  that
\begin{equation}\label{ccc}7(\alpha+\beta-1)+10\alpha(\gamma+\delta-2)+4\ell+\ell'+2<0.\end{equation}
In view of $\alpha+\beta\geqslant1$, it follows from \eqref{ccc} that $\gamma+\delta\leqslant1$, or equivalently, $(\gamma, \delta)\in\{(0, 0), (0, 1), (1, 0)\}$. Since  $\alpha+\beta\leqslant4$ and $\beta+\gamma+2\delta=|S_2|$, one can easily deduce   from \eqref{ccc} that  $\beta=\delta=1$,  $\gamma=\ell=0$ and $\alpha\in\{1, 2\}$. Moreover, if $\alpha=1$,   then it follows from \eqref{ccc} that    $\ell'=0$ and hence  $|S_1\cap\mathscr{B}|=1$. Now,    in  both cases $\alpha=1$ and $\alpha=2$, the structure of $Z$   forces    $F$ to  be   updated  to $Z$  during   the  procedure,   a contradiction.

We next  assume that $S\subseteq V(F)$. From    our    procedure and Fact 1, we observe  that $N_F(v)\in\mathscr{A}$ for any $v\in\mathscr{B}$.  This yields     $S\cap\mathscr{B}=\varnothing$. Hence,     there are $A_1, A_2, A_3, A_4\in\mathscr{A}$ such that $x\in A_1$, $y\in A_2$ and $S_2=A_3\cup A_4$. Note that
there exist  two edges between  $P$ and $Q$ for any $(P, Q)\in\{(A_1, A_3), (A_1, A_4), (A_2, A_3), (A_2, A_4)\}$. According to the procedure, each  $X\in\mathscr{A}$ is connected to exactly one of the elements of $\mathscr{A}$ generated prior
to $X$. This property  contradicts   the cyclic connection  between    $A_1,A_2,A_3,A_4$.

We have established  the claim   and so  the theorem is concluded.
\end{proof}

\begin{remark}
An easy but weak upper bound on $p_c(n; K_{2,t})$ can be found as follows. 
If a graph $G$ has  a copy of $\mathpzc{G}_{t-1}(u; u_1, \ldots, u_{t-2})$ as a subgraph, then one can easily see that a copy of $K_{t-1,t-1}$ is contained in $\widehat{G}$.
Therefore,   a threshold for the existence of    $\mathpzc{G}_{t-1}(u; u_1, \ldots, u_{t-2})$ in   $\mathbbmsl{G}(n, p)$  gives an upper bound on 
$p_c(n; K_{2,t})$.  This shows that   $p_c(n; K_{2,t})=\mathrm{O}(n^{-(t-1)/(2t-4)})$ using   Theorem \ref{bolo}.
\end{remark}

\section{The  lower  bound}

In this     section,     we give   a lower  bound on  $p_c(n; K_{2,t})$. In  \cite[Proposition\,25]{bal},  Balogh, Bollob\'as and Morris provided    a lower  bound   on  $p_c(n; H)$ for any $H$. According to  their result,  $p_c(n; K_{2,t})=\Omega(n^{-(t+1)/(2t-2)})$.  An improvement is given    in the following theorem.

\begin{thm}\label{lower}
For any fixed integer $t\geqslant4$, $$p_c(n; K_{2,t})=\Omega\hspace{-1mm}\left(n^{-\tfrac{t}{2t-3}}\right).$$
\end{thm}

\begin{proof}
If $G\thicksim\mathbbmsl{G}(n, p)$ with  $p\ll n^{-{t}/(2t-3)}$, then  Theorems  \ref{bolo} together with     the union bound theorem   yield   that $G$ with high probability contains no  bounded subgraph $H$ with $m(H)\geqslant (2t-3)/t$. So,  in order to prove  the theorem, it suffices to  show  that any graph     with no     bounded subgraph $H$ satisfying   $m(H)\geqslant(2t-3)/t$    does not percolate  in  $K_{2,t}$-bootstrap process.

Fix a graph   $G$    without any   bounded subgraph $H$ with  $m(H)\geqslant(2t-3)/t$.
Consider  a maximal family   $\mathscr{F}=\{F_1, \ldots, F_\ell\}$ of vertex disjoint copies of $K_{2,t-1}$ in $G$. Denote the vertex bipartition of $F_i$ by   $\{a_{i1}, a_{i2}\}$ and   $\{b_{i1}, \ldots,  b_{i, t-1}\}$. Denote  by $G'$  the   graph obtained from $G$ by joining   $a_{i1}$ to all vertices in $N_G(a_{i2})\setminus N_G(a_{i1})$ and   $a_{i2}$ to all vertices in $N_G(a_{i1})\setminus N_G(a_{i2})$  for    $i=1, \ldots, \ell$.  We claim that  $\widehat{G}=G'$.  Since the graph obtained from $K_{2,t-1}$ by adding one edge has density $(2t-1)/(t+1))>(2t-3)/t$, our assumption on $G$ concludes that $G'$ is not a complete graph. So,    the claim yields  that $G$ does not percolate in  $K_{2,t}$-bootstrap process.

In order to prove  the claim,   it is sufficient   to show  that    there exists no pair $\{x, y\}\notin\{\{a_{11}, a_{12}\}, \ldots, \{a_{\ell1}, a_{\ell2}\}\}$ so that  $|N_{G'}(x) \cap N_{G'}(y)|\geqslant  t-1$. By   contrary, suppose that there exists such a  pair   $\{x, y\}$. Let $S_1=\{x, y\}$ and $p_i=|\{a_{i1},a_{i2}\}\cap S_1|$ for any $i$. By the  assumption, $p_i\in \{0,1\}$. Further, fix a subset  $S_2\subseteq N_{G'}(x)\cap N_{G'}(y)$ such that $|S_2|\in\{t-1, t\}$ and $q_i=|\{a_{i1},a_{i2}\}\cap S_2|\in\{0, 2\}$ for any $i$. Put  $S=S_1\cup S_2$ and  $k=|S|$.   Assume that
\begin{align*}
&\alpha=|\{i \, | \, p_i=1\}|, \\
&\beta=|\{i \, | \, q_i=2\}|, \\
&\gamma=\big|\big\{i \, \big| \,      p_i=q_i=0   \text{ and there exists $j$ with } b_{ij}\in S\big\}\big|,\\
&\lambda=\big|\big\{b_{ij} \, \big| \,    b_{ij}\in S_1 \text{ and } p_i=1\big\}\cup\big\{b_{ij} \, \big| \,   b_{ij}\in S_2 \text{ and } q_i=2\big\}\big|,\\
&\mu=\big|\big\{b_{ij} \, \big| \,     b_{ij}\in S_1 \text{ and } q_i=2\big\}\cup\big\{b_{ij} \, \big| \,  b_{ij}\in S_2 \text{ and } p_i=1\big\}\big|,\\
&\nu=\big|\big\{b_{ij} \, \big| \,      p_i=q_i=0   \text{ and } b_{ij}\in S\big\}\big|.
\end{align*}
Note that  $\alpha+\beta+\gamma\geqslant1$ and $\gamma\leqslant\nu$. Let $$H=G\Bigg[S\cup\hspace{-2mm}\bigcup_{S\cap V(F_i)\neq\varnothing}\hspace{-3mm}V(F_i)\Bigg].$$
It is easy to see that $$|V(H)|=(\alpha+\beta+\gamma)(t+1)+k-\alpha-2\beta-\lambda-\mu-\nu$$ and $$|E(H)|\geqslant 2(\alpha+\beta+\gamma)(t-1)+2(k-\beta-2)-\mu.$$
The condition $k\leqslant t+2$ implies that  $|V(H)|$  is bounded and so  $m(H)<(2t-3)/t$ by the assumption on $G$. It follows from $d(H)<(2t-3)/t$ that
$$t(\alpha+\beta-\gamma+2\lambda+\mu+2\nu-4)<3(\beta-\gamma+\lambda+\mu+\nu-k),$$ which  can be rewritten as
$$(t-3)\big((\alpha+\beta+\gamma-1)+\mu\big)+(2t-3)\big((\nu-\gamma)+\lambda\big)+3\big(\alpha+\gamma+\big(k-(t+1)\big)\big)<0.$$
We have reached a contradiction,   since   the left hand side of the inequality above  is nonnegative. This establishes  the claim, as required.
\end{proof}

\section{Concluding remarks}

In this paper, we have determined an upper bound for the  threshold  of   $K_{2,t}$-bootstrap percolation by proposing a subgraph whose existence forces  the graph to  percolate. Note  that if Question \ref{que} has an affirmative answer,  then    \cite[Theorem\,5.4]{fri} implies that  $K_{2,t}$-bootstrap percolation  has a  coarse threshold.    Question \ref{que} has been answered positively  in the case   $t=4$. We think  that    our  approach can be used to resolve     Question \ref{que}  for $t=5$. However, it    does not seem  promising  for $t\geqslant6$.

\end{document}